\documentclass{article}
\usepackage{graphicx, amsmath, amsthm, amssymb, color,tikz}
\usepackage{titlesec}
\usepackage[]{overpic}
\usepackage{geometry}
\titleformat{\subsection}[runin]{\normalfont\normalsize\bfseries}{\thesubsection}{1em}{}

\newtheorem{theorem}{Theorem}
\newtheorem{lemma}[theorem]{Lemma}
\newtheorem{proposition}[theorem]{Proposition}

\theoremstyle{definition}

\title{Distinguishing closed 4-manifolds by slicing}
\author{Tye Lidman, Lisa Piccirillo}

\newcommand{\mfs}{\mathfrak{s}}
\newcommand{\mft}{\mathfrak{t}}
\newcommand{\cptwo}{\mathbb{C}P^2}
\newcommand{\cptwobar}{\overline{\cptwo}}

\definecolor{lblue}{HTML}{A1BAC2}
\definecolor{ldblue}{HTML}{637D94}
\definecolor{lgreen}{HTML}{AEBC7D}
\definecolor{lred}{HTML}{BB6A5E}

\begin{document}

\maketitle

\begin{abstract}
One approach to produce homeomorphic-but-not-diffeomophic closed 4-manifolds $X,X’$ is to find a knot which is smoothly slice in $X$ but not in $X’$. This approach has never been run successfully. We give the first examples of a pair of closed 4-manifolds with the same integer cohomology ring where the diffeomorphism type is distinguished by this approach. Along the way, we produce the first examples of 4-manifolds with nonvanishing Seiberg-Witten invariants and the same integer cohomology as $\mathbb{C}P^2\#\overline{\mathbb{C}P^2}$ which are not diffeomorphic to $\mathbb{C}P^2\#\overline{\mathbb{C}P^2}$. We also give a simple new construction of a 4-manifold which is homeomorphic-but-not-diffeomorphic to $\mathbb{C}P^2\#5\overline{\mathbb{C}P^2}$. 
\end{abstract}

\vspace{10pt}
One strategy to disprove the smooth four-dimensional Poincar\'e conjecture is to find a homotopy sphere $W$ such that there is a knot that is smoothly slice in one of $W$ or $S^4$, but not in the other. (Throughout, we will say a knot $K$ in $S^3$ is slice in a closed manifold $X$ if it is slice in the 4-manifold obtained from $X$ by removing an open ball). 
This argument, which dates back to Casson, seems difficult to run in practice, even though Rasmussen's $s$-invariant (and its generalizations) could provide the obstruction to slicing in $S^4$ \cite{Rasmussen}. In fact, while it is well known that pairs of homeomorphic-but-not-diffeomorphic 4-manifolds abound, to date no such \textit{exotic} pairs have been distinguished by this slicing argument.\footnote{Closed exotica has however been detected by versions of this argument which add hypotheses on the homology class of the slice disk, see \cite{MMP}.}

To develop tools for distinguishing 4-manifolds by slicing, one might turn to the easier problem of smoothly distinguishing pairs of smooth closed 4-manifolds which perhaps are not homeomorphic, but at least have the same cohomology ring.  To the authors' knowledge, even this has not been done. In this easier setting, the analogue of disproving the Poincare conjecture becomes
finding an integer homology sphere $W$ such that there is a knot which is slice in one of $W$ or $S^4$, but not in the other. While we cannot solve this maximally small version of the easier problem, we do perhaps the next best thing.
\begin{theorem}\label{thm:slice}
There are spin rational homology four-spheres $B$ and $W$ with $H_1= \mathbb{Z}/2$ such that the figure-eight knot is slice in $B$ but not in $W$.
\end{theorem}
Our $B$ is the simplest closed smooth 4-manifold with $H_1=\pi_1=\mathbb{Z}/2$, and our $W$ has the same cohomology ring as $B$. 


In our setting, the knot in question is slice in the less complicated four-manifold. This cannot happen in the maximally small versions of this problem, since any knot which is slice in $S^4$ is automatically slice in any other homology sphere. We also note that the figure-eight knot is not particularly special; the theorem holds for any strongly negatively amphichiral knot with non-trivial Arf invariant and four-ball genus equal to 1.  

The methods of our construction of $W$, described in Section~\ref{sec:constructions}, can also be used to produce other new 4-manifolds with simple cohomology rings. In the late aughts,  \cite{AkhmedovSymplectic, FSP} gave examples of nonstandard spin symplectic 4-manifolds with the cohomology ring of $S^2\times S^2$. Their examples are presumably not simply connected, in particular their homeomorphism type remains unknown. We re-prove this (see Theorem~\ref{thm:S2timesS2} below) and establish a non-spin analogue, which we believe is new.



\begin{theorem}\label{thm:1+1}
There exists a nonstandard cohomology $\mathbb{C}P^2 \# \overline{\mathbb{C}P^2}$ with non-vanishing Seiberg-Witten and Heegaard Floer four-manifold invariants.  
\end{theorem}

We can also apply our methods to produce honest exotic pairs. Combining the constructions of this paper with a construction from earlier work with Levine \cite[Section 6]{LLP}
we obtain medium-sized exotica. 

\begin{theorem}\label{thm:1+5}
There exists a 4-manifold which is homeomorphic-but-not-diffeomorphic to $\cptwo \#_5 \cptwobar$.
\end{theorem}

We note that this homeomorphism type is already well-known to support infinitely many smooth structures (originally \cite{JSS}, see also \cite{FSP} and others). The novelty here is the construction, which the reader may or may not find simpler than others. 

\section*{Organization} In Section~\ref{sec:constructions} we build the key four-dimensional piece, $V$, using Luttinger surgeries on a genus 2 surface bundle over a punctured torus.  We use $V$ to construct the rational homology sphere $W$.  In Section~\ref{sec:sliceness} we establish that the figure-eight knot is not slice in $W$.  In Section~\ref{sec:other}, we prove Theorem~\ref{thm:1+1} and Theorem~\ref{thm:1+5}.

\section*{Acknowledgements} The first author is supported in part by NSF grant DMS-2105469.  He thanks the Department of Mathematics at the University of Texas at Austin for its hospitality.  The second author is supported in part by a Sloan Fellowship, a Clay Fellowship, and the Simons collaboration “New structures in low-dimensional topology”.  We thank Dani Alvarez-Gavela, \.{I}nan\c{c} Baykur, John Etnyre, Adam Levine, Steven Sivek, and Mike Usher for helpful discussions.

\section{Constructions}\label{sec:constructions}
The manifold $B$ is the Kawauchi manifold c.f. \cite{Kawauchi}, which has the following description.  Take the 0-trace on $4_1$ and quotient by the free orientation-reversing involution on the boundary, $S^3_0(4_1)$, coming from the strongly negatively amphichiral symmetry of $4_1$.  
(By recent work of Levine \cite{Levine}, $B$ is independent of the strongly negatively amphichiral knot in the construction.) It is straightforward to check that $B$ satisfies the conditions in Theorem~\ref{thm:slice}.  Further, since the 0-trace on $4_1$ embeds in $B$ by construction, we see that $4_1$ is slice in $B$.  (For an earlier construction of a rational homology sphere with $\pi_1 = \mathbb{Z}/2$ where the figure-eight knot is slice, see \cite{FS}.  We have chosen Kawauchi's description since it is closer in nature to the new rational homology sphere we build subsequently.)

The manifold $W$ will be built using a few steps.  Here is an outline.  First, we build a genus 2 surface bundle over a once-punctured torus with boundary $S^3_0(Q)$, where $Q$ is the square knot.  We then perform some Luttinger surgeries to create a symplectic homology $S^2 \times D^2$; this is our key piece $V$.  Quotienting by a free involution on $S^3_0(Q)$ gives $W$.  Now we do this concretely.

\begin{figure}
    \centering
\begin{tikzpicture}
\node[anchor=south west] at (0,0) {\includegraphics[scale=0.2]{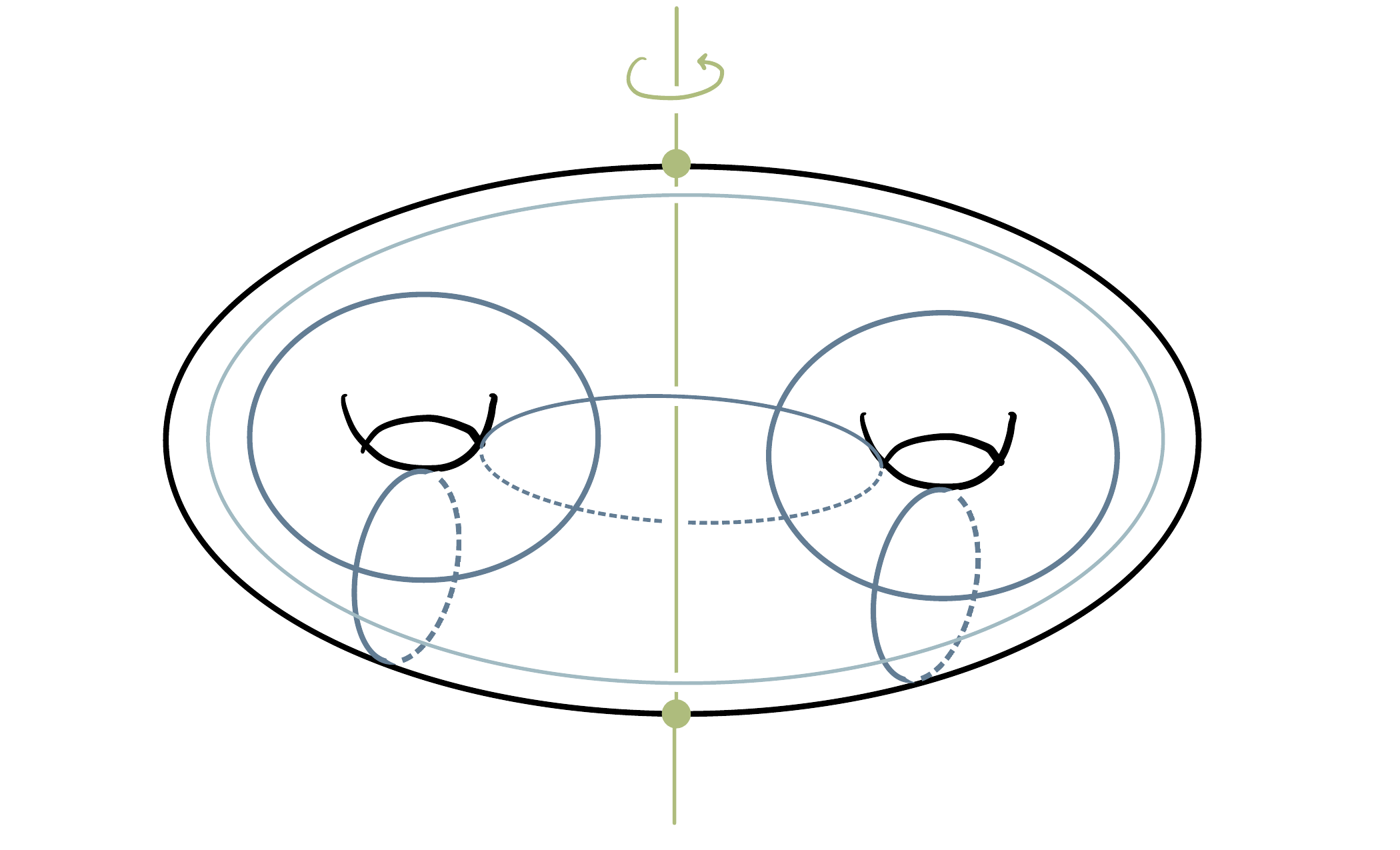}};
\draw[color=ldblue] (2.1,.8) node {$a$};
\draw[color=ldblue] (2.9,3) node {$b$};
\draw[ color=ldblue] (3.8,2.7) node {$c$};
\draw[ color= ldblue] (5.3,3) node {$d$};
\draw[color=ldblue] (5.0,.8) node {$e$};
\draw[color=lgreen] (4,4.4) node {$\phi$};
 \draw (1,1) node {$F$};
 \draw[color=lblue] (2.6,3.8) node {$z$};
 \end{tikzpicture}    
\caption{}
\label{fig:S}
\end{figure}

Consider a genus 2 surface $F$ equipped with curves $a,b,c,d, e$ configured as in Figure~\ref{fig:S}. It is well-known that 0-surgery on the square knot $Q$ is fibered with fiber $F$ and monodromy conjugate to $abd^{-1}e^{-1}$, where we write a letter to mean a positive Dehn twist along that letter.  Let $\phi$ be the involution on $F$ shown in Figure~\ref{fig:S}, which exchanges $a$ and $e$, $b$ and $d$, and fixes $c$.  Note that $abd^{-1}e^{-1} = (ab)\phi (ab)^{-1} \phi^{-1}$, i.e. $abd^{-1}e^{-1}$ is a commutator in the mapping class group of $F$, and hence $S^3_0(Q)$ bounds a genus 2 fiber bundle with base a once-punctured torus, denoted $R$.  (If $\alpha, \beta$ are a basis for the fundamental group of the punctured torus, then $R$ is specified by having say monodromy $ab$ along $\beta$ and monodromy $\phi$ along $\alpha$.)  Note that $R$ is a symplectic four-manifold with $\chi(R) = 2$ and the canonical class evaluates to $\pm 2$ on a fiber \cite{Thurston}.  

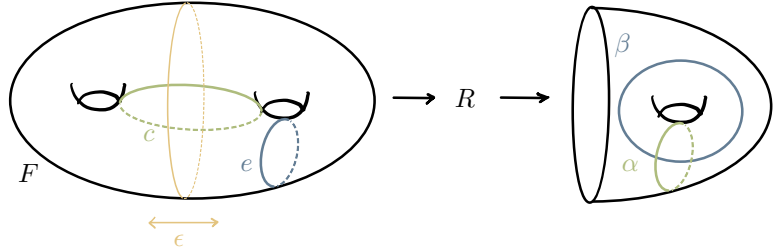
\begin{figure}
\centering
\begin{tikzpicture}
\node[anchor=south west] at (0,0) {\includegraphics[trim=0 150 0 150, clip, scale=.3]{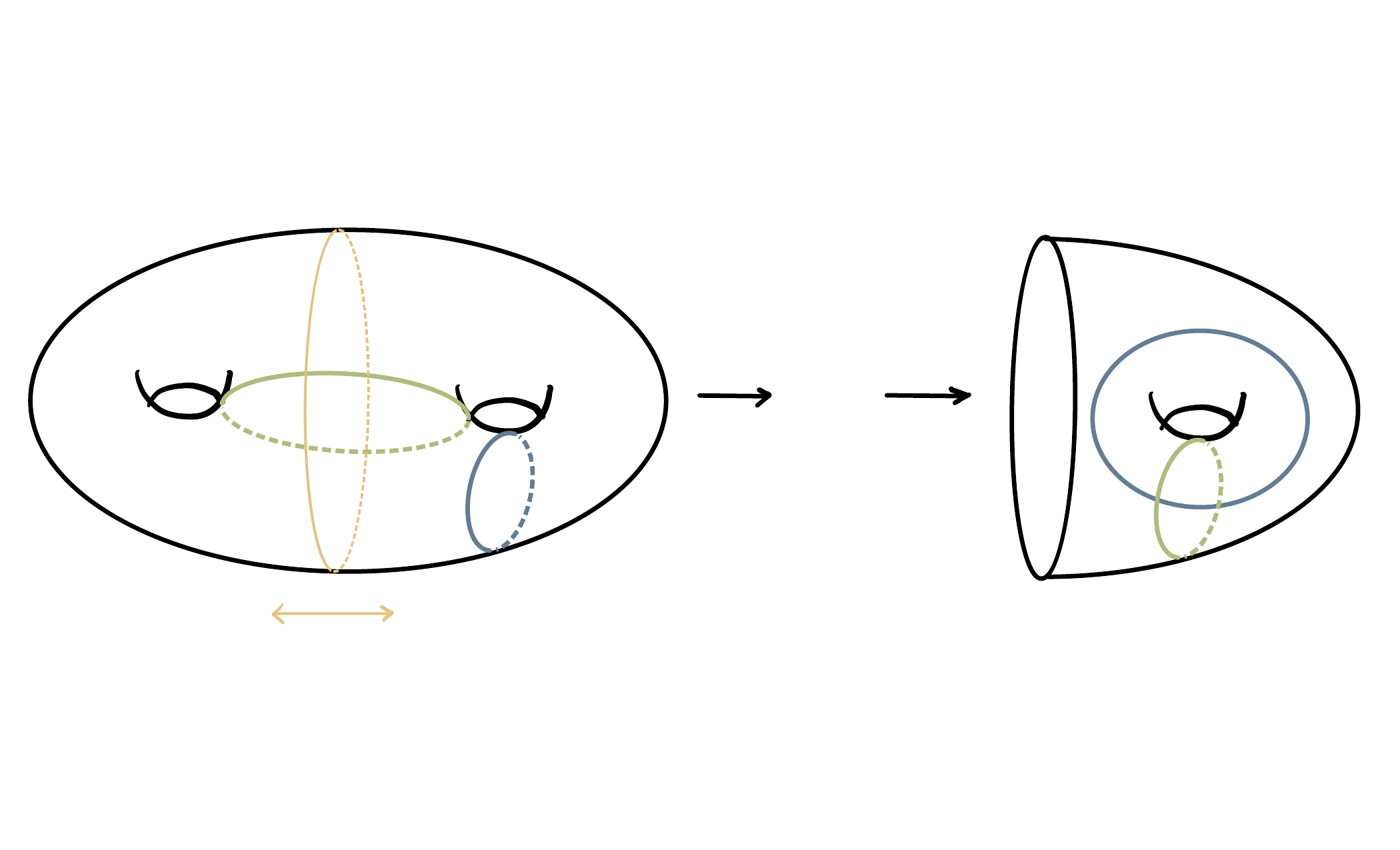}};

\definecolor{lyellow}{HTML}{E3C480}
\node[text=ldblue] at (3.5,1.1) {$e$};
\node[text=lgreen] at (2.2,1.5) {$c$};
\node[text=lyellow] at (2.6,.1) {$\epsilon$};
\node[] at (.6,1) {$F$};

\node[text=ldblue] at (8.5,2.7) {$\beta$};
\node[text=lgreen] at (8.6,1.1) {$\alpha$};
\node[] at (6.4,2) {$R$};
\end{tikzpicture}
\caption{R is a genus 2 surface bundle over a puncured torus. The surgery torus $T_\beta$ is the sub-bundle given by the restriction to the two curves marked in blue, and $T_\alpha$ to the curves in green. The orientation reversing involution $\epsilon$ on $F$ is marked in yellow.}
\label{fig:bundle}
\end{figure}

Since the monodromy $ab$ fixes $e$ pointwise, we have a torus $T_\beta$ in $R$ given by the sub-bundle restricted to $e$ in the fiber and $\beta$ in the base.  Similarly, $\phi$ fixes $c$ setwise and with orientation, and so we have a torus $T_\alpha$ from the restriction to $c$ in the fiber and  $\alpha$ in the base.  Notice that there is an area form on the fiber preserved by both monodromies $ab$ and $\phi$; this induces a symplectic form on $R$ for which $T_\alpha$ and $T_\beta$ are Lagrangian.  We will need a parametrization of these tori as submanifolds of $R$, which we can get in the following way; let $\beta'$ and $\alpha'$ be curves on $T_\beta$ and $T_\alpha$ which project to $\beta$ and $\alpha$. (For curves in the fiber, we will use the inclusion to think of them as curves in $R$, and not give them a new name). Thus $T_\beta$ is parametrized by $(e,\beta')$ and $T_\alpha$ by $(c,\alpha')$. Both $T_\alpha$ and $T_\beta$ come with a Lagrangian framing. Define $\alpha'_\partial$ and $\beta'_\partial$ to be pushoffs of $\alpha'$ and $\beta'$ into the boundaries of a neighborhood of $T_\alpha$ and $T_\beta$ using this framing. We now perform a $+1$-Luttinger surgery along each torus with directions $\beta'$ and $\alpha'$ respectively; in particular, we remove a neighborhood of $T_\beta$ (respectively $T_\alpha$) and reglue in $T^2 \times D^2$ so that $\partial D^2$ goes to $\mu_{T_\beta}+\beta_\partial'$ (respectively $\mu_{T_\alpha}+\alpha_\partial'$).  Call the result\footnote{We note that since we parametrized  the surgery tori somewhat arbitrarily, $V$ is not well-defined per se. Since any such parametrization yields a $V$ for which all claims of the rest of the paper hold, we are content with this ambiguity.} $V$, which necessarily still has $\chi(V) = 2$.  We will argue momentarily that $V$ is a spin integer homology $S^2$; assume that for now, and note that $V$ is symplectic with canonical class evaluating to $\pm 2$ on a copy of $F$ away from where the Luttinger surgeries happened.  

We now define $W$ to be $V/\sigma$, where $\sigma$ is the free, orientation-reversing boundary automorphism described in the following lemma: 

\begin{lemma}\label{lem:gluing-symplectic} 
    There is a free orientation-reversing bundle isomorphism $\sigma$ on $S^3_0(Q).$ 
\end{lemma}
\begin{proof}
    To begin, notice that while we have described $S^3_0(Q)$ as a genus 2 surface bundle over $S^1$ with monodromy $abd^{-1}e^{-1}$, there is a bundle isomorphism to a genus 2 surface bundle over $S^1$ with monodromy $abe^{-1}d^{-1}$, where the isomorphism comes from conjugating the monodromy by $e^{-1}$. We will work with this latter description. 

    Now think of the composed mapping torus for the composition $(ab)\circ (e^{-1}d^{-1})$. If we think of the monodromy factors occuring respectively at $\pm 1$ in the $S^1$ base, then we obtain an orientation reversing bundle isomorphism $\sigma$ on $S^3_0(Q)$ by rotation by $\pi$ in the base and $\epsilon$ reflection in the fiber, where $\epsilon$ is described in Figure \ref{fig:bundle}. 
\end{proof}

In order to show that $W$ has the desired properties from Theorem~\ref{thm:slice}, we will need to analyze $V$ more carefully.  
\begin{lemma}
    The manifold $V$ constructed above is a spin 4-manifold with $H_*(V)\cong H_*(S^2).$ Further, $\pi_1(V)$ is normally generated by $\pi_1(F)$, where $F$ is a fiber of $R$ away from the Luttinger surgeries. 
\end{lemma}\label{lem:Valgtop}

\begin{proof}
We begin with the spin and homology assertions. We will prove that $H_1(V)=0$. Since $\partial V$ is connected, it then is routine to check (using Poincare-Lefshetz duality and universal coefficients) that $H_3(V)=0$ and $H_2(V)$ is free. The fact that $\chi(V)=2$ then implies that $H_2(V)=\mathbb{Z}$, which one can argue has to be generated by the fiber $F$. Since the fiber has self intersection zero, $V$ is spin. 

To compute $H_1(V)$ we first compute $H_1(R\setminus (T_\alpha \cup T_\beta))$. We will commit the standard abuse of referring to a homology class by a curve representing it. By a standard argument, we know that $H_1(R\setminus (T_\alpha \cup T_\beta))$ is generated by $H_1(R)$ and $\{\mu_{T_\alpha}, \mu_{T_\beta}\}$. We know that $H_1(R)$ is generated by the $H_1$ of a fiber and $\{\alpha', \beta'\}$, curves which project to generate $H_1$ of the base. Since $H_1(F)$ dies in $H_1(S^3_0(Q))$, and hence in $H_1(R)$, $H_1(R)$ is just generated by $\{\alpha',\beta'\}$.  Notice also that there are geometric dual tori $T_\alpha^*$ and $T_\beta^*$ to $T_\alpha$ and $T_\beta$. For example, one can take the tori parametrized as $T_\alpha^*=(b, \beta'')$ and $T_\beta^*=(z,\alpha'')$, where $z$ is the gray curve given in Figure~\ref{fig:S} and $\{\beta'',\alpha''\}$ 
project to $\beta$ and $\alpha$. This implies that $\mu_\alpha$ is freely homotopic to $[b,\beta'']$ and $\mu_\beta$ is freely homotopic to $[z,\alpha'']$ in $R \setminus (T_\alpha \cup T_\beta)$; in particular both $\mu_{T_\alpha}$ and $\mu_{T_\beta}$ are trivial in $H_1(R\setminus (T_\alpha \cup T_\beta))$. Thus $H_1(R\setminus (T_\alpha \cup T_\beta))$ is still generated by $\{\alpha'_\partial,\beta'_\partial\}$. Finally, we obtain $V$ by filling $R\setminus (T_\alpha \cup T_\beta))$ with slopes $\mu_{T_\beta}\cdot\beta'_\partial$ and $\mu_{T_\alpha}\cdot\alpha'_\partial$; these fillings contribute relations which kill $\alpha_\partial'$ and $\beta_\partial'$. Hence $H_1(V)=0$.  

For the $\pi_1$ assertion, it is again standard to check that $\pi_1(V)$ is normally generated by $\alpha'_\partial, \beta'_\partial,\mu_{T_\alpha}, \mu_{T_\beta}$ and curves in the fiber $F$. We argued above that 
\begin{enumerate}
    \item in $R\setminus (T_\alpha \cup T_\beta)$, $\mu_{T_\alpha}$ is freely homotopic to $[b,\beta'']$ and $\mu_{T_\beta}$ is freely homotopic to $[z,\alpha'']$, 
    \item the filling meridians give relations $\mu_{T_\beta}\cdot\beta_\partial'=1$ and $\mu_{T_\alpha}\cdot\alpha_\partial'=1.$
\end{enumerate} We will argue that $\pi_1(V)/\langle\langle\pi_1(F)\rangle\rangle=0$. It follows from Item 1 that both $\mu_{T_\alpha}$ and $\mu_{T_\beta}$ die in $\pi_1(V)/\langle\langle\pi_1(F)\rangle\rangle$. In then follows from Item 2 that both $\alpha_\partial'$ and $\beta_\partial'$ die in $\pi_1(V)/\langle\langle\pi_1(F)\rangle\rangle$. Since we have killed all normal generators of $\pi_1(V),$ we get that $\pi_1(V)/\langle\langle\pi_1(F)\rangle\rangle=0$, as desired.
\end{proof}

\begin{lemma}\label{lem:section}
    There exist a pair of disjoint embedded surfaces $\Gamma,\Gamma',$ each generating $H_2(V,\partial V)$, whose boundaries in $\partial V$ are the subundles over $S^1$ given by restricting the fiber to the two fixed points of $\phi$.   
\end{lemma}
\begin{proof}
    Begin by considering $R$; here we can see that the fixed points of $\phi$ on $F$ are fixed by the entire monodromy of the bundle over the punctured torus. Hence these points give rise to disjoint sections $\Gamma,\Gamma'$ with  boundaries as described in the lemma statement. Since we can take our Luttinger surgeries to miss these two sections, $\Gamma, \Gamma'$ survive into $V$, where they remain disjoint. Since both $\Gamma, \Gamma'$ have one point of intersection with $F$, it is straightforward to check that either one generates $H_2(V,\partial V)$.
\end{proof}

With the algebraic topology of $V$ in hand, we are ready to check that $W$ has the desired algebraic topology.  
\begin{lemma}\label{lem:Wproperties}
The manifold $W$ is a spin rational homology sphere with $H_1(W) = \mathbb{Z}/2$.  
\end{lemma}
\begin{proof}
Notice that $W$ is 2-fold covered by $V\cup_\sigma V$, which has no $H_1$. This shows $H_1(W)=\mathbb{Z}/2$.  Since the double cover of $W$ has Euler characteristic 4, we see that $b_2(W) = 0$, and hence $W$ is a rational homology sphere.  

To see that $W$ is spin, we just need to check that $w_2(TW)$ vanishes. By the Wu formula, $w_2(TW)$ is characteristic on $H_2(W,\mathbb{Z}/2)$. As such it suffices to show that the intersection form on $H_2(W,\mathbb{Z}/2)$ is even. We will prove now that this intersection form is the hyperbolic form.

By another Euler characteristic argument, we see that $H_2(W,\mathbb{Z}/2)$ is $\mathbb{Z}/2 \oplus \mathbb{Z}/2$. We will demonstrate a pair of 2-cycles in $H_2(W,\mathbb{Z}/2)$ which both have self-intersection 0, and which have pairwise intersection 1. As such, these cycles will form a hyperbolic pair for $H_2(W,\mathbb{Z}/2)$, and we will be done. 

The first 2-cycle is the image of $F$ after the quotient; that $F$ has self-intersection 0 is inherited from $V$. The second cycle will be the image of the generator $\Gamma$ of $H_2(V,\partial V)$ that we set up in Lemma \ref{lem:section}. Note that $\Gamma$ has intersection 1 with $F$ in $V$, and this will be preserved in the quotient. We need to check that $\Gamma$ gives a cycle in $H_2(W,\mathbb{Z}/2)$, and compute its self-intersection. Both of these claims follow from the observation that the quotient map $\sigma$ acts as an orientation and component-preserving free involution on the boundaries of $\Gamma$ and $\Gamma'$, so $\Gamma$ descends to a closed (non-orientable) surface. 
\end{proof}

\section{Obstructing sliceness}\label{sec:sliceness}
In this section, we will prove that $4_1$ is not slice in the manifold $W$ constructed above, completing Theorem~\ref{thm:slice}.  
We will obstruct sliceness in $W$ by studying the genus function of the double-cover 
$V \cup_\sigma V$.  Since $V$ is a homology 0-trace with amphichiral boundary, we see $V \cup_\sigma V$ is a closed 4-manifold with $b_2=2$. In fact, this closed manifold is symplectic.

\begin{theorem}\label{thm:S2timesS2}
The closed manifold $Z=V\cup_\sigma V$ is a symplectic cohomology $S^2 \times S^2$ not diffeomorphic to $S^2 \times S^2$. 
\end{theorem}
\begin{proof}
We will first check symplecticness, then cohomology and spinness, and conclude by showing $Z$ is not $S^2\times S^2$.

To see that $Z$ is symplectic, notice that the gluing $\sigma$ respects the fiber structure on $\partial V\cong\partial R$ (see Lemma \ref{lem:gluing-symplectic}). Under such a gluing, we can think of $Z$ as a genus 2-bundle over a genus 2-surface\footnote{Out of the box, $\sigma$ is orientation reversing on the fiber and preserving on the section; to get $R\cup_\sigma R$ to be an oriented surface bundle over an oriented surface we want the opposite. But we can fix this by considering our second copy of $R$ to have orientation given by $(-F, -B)$. Notice that this is still just $R$ (as an oriented manifold), but from this perspective when we glue $R\cup_\sigma R$ we have that $\sigma$ preserves orientation in the fiber and flips it in the sections; resulting in a closed genus 2 surface bundle over a genus 2 surface, as desired.} (namely $R\cup_\sigma R$) to which we have performed four torus surgeries. Since genus 2 surface bundles over surfaces are symplectic \cite{Thurston}, our tori are Lagrangian, and Luttinger surgeries preserve symplecticness \cite{Luttinger}, we have that $Z$ is symplectic. 

It is easy to see that $H_*(Z) \cong H_*(S^2 \times S^2)$.  To see that $Z$ has the same cohomology ring as $S^2 \times S^2$, we need to check that $Z$ is spin.  This is follows from the fact that $Z$ is the double cover of the spin manifold $W$.  

We have seen that $Z$ is a symplectic four-manifold.  To show $Z$ is not diffeomorphic to $S^2 \times S^2$ we reproduce an argument of Akhmedov-Park.  The Kodaira dimension $\kappa$ of a minimal symplectic four-manifold is a diffeomorphism invariant by \cite{LiKodaira}.  This invariant is $-\infty$ if and only if the manifold is rational or ruled, e.g. $S^2 \times S^2$.  A non-trivial genus 2 surface bundle over a genus 2 surface is minimal by asphericity, and thus never rational or ruled. Hence $R\cup_\sigma R$ has $\kappa \neq -\infty$.  Because $\kappa$ is preserved by Luttinger surgeries \cite{HoLi} and $Z$ is minimal (e.g. because it is spin), the Kodaira dimension shows that $Z$ is not diffeomorphic to $S^2 \times S^2$.  This completes the proof.  \end{proof}

\begin{proof}[Proof of Theorem~\ref{thm:slice}]
The knot $4_1$ is slice in $B$ by construction. It remains to prove that $4_1$ is not slice in $W$.  Suppose for a contradiction that it was.  Since $W$ is spin, the Arf invariant obstructs $4_1$ from bounding a nullhomoloogous slice disk $D$ (see e.g. \cite[Theorem 2]{Klug}).   


So $4_1$ would have to bound a slice disk which is non-trivial in homology.  This imples that $X_0(4_1)$ embeds in $W$ with nontrivial inclusion induced map on $H_2$.  Inside of $X_0(4_1)$ is a square-zero torus $T$ obtained by capping a Seifert surface for $4_1$ with the slice disk, and this $T$ is non-trivial in $H_2(W)$.  Since $X_0(4_1)$ is simply-connected, $T$ lifts to 
a torus (still called $T$) which is non-trivial in $H_2(Z)$. We will borrow an argument from \cite[Theorem 1.4]{StipsiczSzabo} to show that this cannot happen.   
They show the following: if $M$ is a symplectic cohomology $S^2 \times S^2$ obtained by taking a genus 2 surface bundle over a genus 2 base where the fiber and section form a hyperbolic pair and doing Luttinger surgery on disjoint Lagrangian tori that miss a fiber and a section, then no non-trivial square-zero class in $H_2(M)$ is represented by a torus. We claim that this is exactly the setting we are working in. The only claim we have not already established is that we can find a fiber and section that form a hyperbolic pair disjoint from the surgery tori; in particular we just need to check that our $R\cup_\sigma R$ has a square $0$ section disjoint from the surgery tori.  In Lemma \ref{lem:section} we already established two disjoint generators $\Gamma, \Gamma'$ of $R$ which are disjoint from the surgery tori and whose boundaries are a pair of circles which are setwise preserved by $\sigma$.  Hence, $\Gamma \cup_\sigma \Gamma$ and $\Gamma' \cup_\sigma \Gamma'$ form disjoint sections.  It follows that $R \cup_\sigma R$ has a square-0 section disjoint from the surgery tori. 
\end{proof}

\section{Other constructions}\label{sec:other}
In this section, we prove Theorems~\ref{thm:1+1} and \ref{thm:1+5}.  The arguments use Floer homology.  As we believe more readers are familiar with Heegaard Floer homology than monopole Floer homology, we have written the arguments in this language.  However, similar arguments can be applied for the so-called small-perturbation Seiberg-Witten invariants using monopole Floer homology (see e.g. \cite[Lecture 5]{SixLectures}, \cite[Section 27]{MonopoleBook}).  

\subsection{Heegaard Floer mixed invariants}
To begin, we quickly review the Heegaard Floer mixed invariants for four-manifolds with $b^+ = 1$ as described in \cite{OSSymplectic}.  Let $M$ be a closed four-manifold with $b^+ = 1$ and $L \subset H_2(M;\mathbb{Q})$ a line described as the span of a square-zero class.  Decompose $M = M_1 \cup_Y M_2$ where the image of $H_2(Y;\mathbb{Q})$ in $H_2(M;\mathbb{Q})$ is $L$. If $\mft$ is a spin$^c$ structure on $M$ which restricts to be non-torsion on $Y$, then we have 
\[
\Phi_{M,L, \mft} = \langle \Psi_{M_1, \mft_1} , \Psi_{M_2, \mft_2} \rangle_{HF_{red}(Y)}
\]
where $\langle \cdot, \cdot \rangle_{HF_{red}(Y)}$ denotes the non-degenerate pairing on $HF_{red}(Y)$ and $\Psi_{M_i, \mft_i}$ denotes the relative invariant, i.e. the projection of the image of 1 under $F^-_{M_i-B^4,\mft_i}: HF^-(S^3) \to HF^-(Y)$ to $HF_{red}(Y)$.\footnote{This projection is well-defined if $\mft|_Y$ is non-torsion.  If one works with $U$-completed coefficients, there is no projection necessary.}     
Ozsv\'ath-Szab\'o establish that $\Phi_{M,L, \mft}$ is an invariant of the triple $(M,L,\mft)$.  While not necessary for this paper, if one is interested in spin$^c$ structures which restrict to be torsion on $Y$, then an analogous invariant can be studied using perturbed coefficients.  In practice, the mixed invariants can depend on the choice of $L$.  
As such, when we try to use these $b^+=1$ mixed invariants to obstruct the existence of diffeomorphisms, we will have to make an additional argument to deal with the ambiguity in choice of line.

\subsection{The proofs of Theorem~\ref{thm:1+1} and \ref{thm:1+5}}
In order to prove Theorem~\ref{thm:1+1} and Theorem~\ref{thm:1+5} we need to compute some mixed  invariants.  For this, we need to understand the relative invariants of our key piece $V$.
Fix an orientation of the genus 2 fiber $F$ in $S^3_0(Q)$.  Let $\mfs_{\pm}$ denote the spin$^c$ structure on $S^3_0(Q)$ with $\langle c_1(\mfs_{\pm}), F \rangle = \pm 2$.

\begin{lemma}\label{lem:non-vanishing}
Let $\mft$ be a spin$^c$ structure on $V$ with $|\langle c_1(\mft), F \rangle| = 2$.  Then, $\Psi_{V,\mft} \neq 0$.     
\end{lemma}
\begin{proof}
As we have seen, there exists a gluing $\sigma: \partial V \to - \partial V$ such that $Z = V \cup_\sigma V$ is symplectic.  Further, the canonical class $\mathfrak{k}$ on $Z$ satisfies $|\langle c_1(\mathfrak{k}), F \rangle| = 2$. Since $H^2(V)=\mathbb{Z}$, we must have $\mathfrak{k} |_V = \mft$ or $\overline{\mft}$. Let $L$ be the line in $H_2(Z;\mathbb{Q})$ generated by $F$.  By \cite{OSSymplectic}, $\Phi_{Z,L,\mathfrak{k}} \neq 0$.  It follows that $\Psi_{V,\mathfrak{k}|_V} \neq 0$.  Together with the fact that $\Psi_{V,\overline{\mft}} \neq 0$ if and only if $\Psi_{V,\mft} \neq 0$ \cite[Theorem 3.6]{OSTriangles}, the result follows.
\end{proof}

\begin{lemma}\label{lem:gluing-independent}
Let $(M_1,\mft_1)$ and $(M_2, \mft_2)$ be two spin$^c$ four-manifolds glued along $S^3_0(Q)$ so that $\mft_1|_{S^3_0(Q)} = \mft_2|_{S^3_0(Q)} = \mfs_{\pm}$.  If $\Psi_{M_1,\mft_1}$ and $\Psi_{M_2,\mft_2}$ are non-zero, then $\Phi_{M_1 \cup M_2, span\{F\}, \mft} \neq 0$ for any $\mft\in Spin^c(M_1 \cup M_2)$ with $\mft |_{M_i} = \mft_i$ for both $i = 1, 2$.
\end{lemma}
\begin{proof}
Note that $HF_{red}(S^3_0(Q), \mfs_\pm) = \mathbb{F}$ since $S^3_0(Q)$ is fibered with fiber surface $F$, which has genus 2 \cite[Theorem 5.2]{OSSymplectic}.  Therefore, two elements pair to be non-zero in $HF_{red}(S^3_0(Q), \mfs_\pm)$ if and only if they are non-zero.  The result follows.
\end{proof}


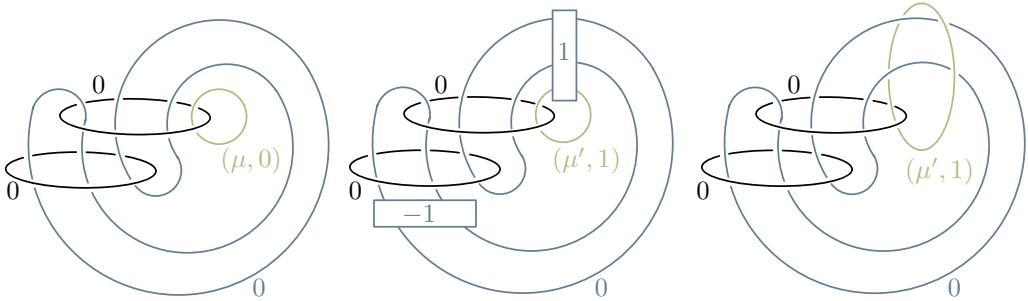
\begin{figure}
\begin{center}
\begin{overpic}[scale=.4, trim = 0in 2.7in 0in 1.9in, clip]{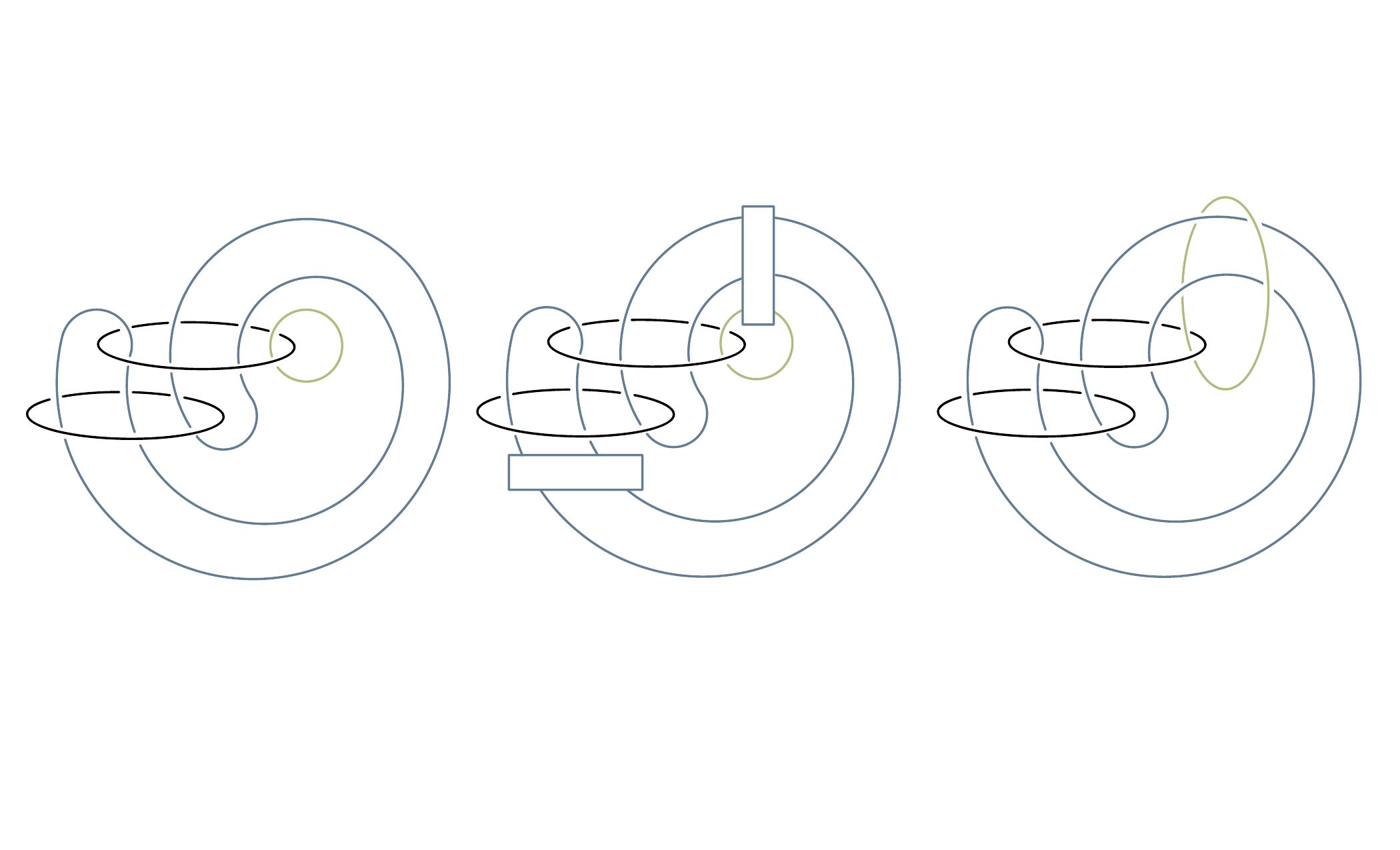}
\put(25, 0){\color{ldblue}$0$}
\put(57, 0){\color{ldblue}$0$}
\put(90, 0){\color{ldblue}$0$}
\put(39, 6.9){\color{ldblue}$-1$}
\put(53.5, 22){\color{ldblue}$1$}
\put(22, 12){\color{lgreen}$(\mu,0)$}
\put(53, 12){\color{lgreen}$(\mu',1)$}
\put(86, 11){\color{lgreen}$(\mu',1)$}
\put(2, 9){$0$}
\put(34, 9){$0$}
\put(66.5, 9){$0$}
\put(10, 19){$0$}
\put(42, 19){$0$}
\put(75, 19){$0$}
  \end{overpic}
  \end{center}
\caption{A self diffeomorphism of $S^3_0(Q)$ which changes the parity of the (homology class of the) meridian.
}\label{fig:spinswap}
\end{figure}

\begin{proof}[Proof of Theorem~\ref{thm:1+1}]
First, we claim that we can cut $Z$ along $S^3_0(Q)$ and reglue by an orientation-preserving diffeomorphism $g$ such that the result is non-spin.  This will be established later in the proof.  Call the result $Z'$.  Since $\Psi_{V, \mft}$ and $\Psi_{V, \overline{\mft}}$ are both non-zero, by Lemma~\ref{lem:gluing-independent} we still get that the invariants of $Z'$ would be non-vanishing for $span\{F\}$.  Furthermore, $Z'$ is now a cohomology $\cptwo \# \cptwobar$.  

We can see that $Z'$ is not diffeomorphic to $\cptwo \# \cptwobar$ or, more generally, $\cptwo \# \cptwobar \# D$ for any homology four-sphere $D$ as follows.  Both $\cptwo \# \cptwobar$ and $\cptwo \# \cptwobar \# D$ each admit two square-zero lines, each represented by a separating $S^2 \times S^1$. Since $HF_{red}(S^2 \times S^1) = 0$ in any non-torsion spin$^c$ structure, we have vanishing mixed invariant for any choice of line.

    Now we construct the regluing homeomorphism $g$. To set up, note that since $V$ is a homology 0-trace, we know that the boundary of a generator $\gamma$ of $H_2(V,\partial V)$ is a generator of $H_1(\partial V)$. Note that the meridian $\mu$ in $S^3_0(Q)$ represents this boundary. We will show that there is a homeomorphism $f:\partial V\to \partial V$ with which sends the framed\footnote{Our integer framing convention is to forget the rest of the diagram and reference the Seifert framing.} meridian $(\mu,0)$ to the framed curve $(\mu',1)$ demonstrated in the right frame of Figure \ref{fig:spinswap}. There is a framed homology $A$ from $(\mu',0)$ to $(\mu,0)$ in $S^3_0(Q)$. We already saw that $V\cup_\sigma V$ is even; in particular if we take a surface $\Gamma$ in $V$ representing $\gamma$ then $\Gamma\cup_\mu\Gamma$ must be an even framed surface in $V\cup_\sigma V$. Therefore, $\Gamma\cup_A\Gamma$ in $V\cup_{f\circ \sigma}V$ is odd. The claim follows by taking $g=f\circ \sigma$. 

    It remains to demonstrate the homeomorphism $f$. We have done so in Figure \ref{fig:spinswap}. The first move is a Gluck twist on both black 0-framed unknots, and the second move is an isotopy. 
\end{proof}

\begin{proof}[Proof of Theorem~\ref{thm:1+5}]
We construct an exotic $\cptwo \#_5 \cptwobar$ as follows.  In \cite[Proposition 6.9]{LLP}, we constructed a homotopy $X_0(Q) \#_4 \cptwobar$, which we will denote as $D$, equipped with spin$^c$ structures $\mathfrak{u}_\pm$ such that $\Psi_{D,\mathfrak{u}_\pm}$ are non-zero in $HF_{red}(S^3_0(Q), \mfs_\pm)$.  
  Let $A = V \cup D$ for any some choice of gluing.  Because $\pi_1(D) = 0$ and $F$ normally generates $\pi_1(V)$ by Lemma \ref{lem:Valgtop}, we get that $\pi_1(A) = 0$.  Since $b_2(A)=6$ and $\sigma(A)=-4$, $A$ is homeomorphic to $\cptwo \#_5 \cptwobar$ by Freedman \cite{Freedman}.  By Lemma~\ref{lem:gluing-independent}, we get that there are $\mft\in Spin^c(A)$ such that $\Phi_{A, span\{F\},\mft}$ are non-vanishing.  Since the Ozsv\'ath-Szab\'o mixed invariants can depend on the choice of line $L$, it is not a priori clear that this implies that $A$ is exotic; we need an additional argument.

We will make use of Wall's theorem that for $\cptwo\#5\cptwobar=S^2\times S^2\#3\cptwobar$, every automophism of $H_2$ can be realized by a diffeomorphism \cite{Wall}. To set up, note that the intersection form of $A$ is congruent to $\begin{pmatrix} 0 & 1 \\ 1 & n \end{pmatrix} \oplus 4 \langle -1 \rangle$; write an ordered basis for which this is the form as $E_1, \ldots, E_6$.  Notice also that in $\cptwo \#_5 \cptwobar$, there is a basis $E'_1, \ldots, E'_6$ for $H_2$ with exactly the same intersection form, and such that every generator $E'_i$ is represented by a sphere.  Now, if $A$ was diffeomorphic to $\cptwo \#_5 \cptwobar$, Wall's theorem would tell us that there would be a diffeomorphism sending the $E_i$ classes to the $E'_i$ classes.  By pulling back those spheres along the diffeomorphism, the fiber class $E_1=[F]\in H_2(A)$ would be represented by a square zero sphere.  But since $HF_{red}(S^2 \times S^1) = 0$, this would force $\Phi_{A, span\{F\},\mft} =0$ for all suitable spin$^c$ structures $\mft$. We have already shown this is not true. 
\end{proof}

\bibliographystyle{alpha}
\bibliography{references}

\begin{thebibliography}{MMP24}

\bibitem[Akh08]{AkhmedovSymplectic}
Anar Akhmedov.
\newblock Construction of symplectic cohomology {$S^2\times S^2$}.
\newblock In {\em Proceedings of {G}\"okova {G}eometry-{T}opology {C}onference 2007}, pages 36--48. G\"okova Geometry/Topology Conference (GGT), G\"okova, 2008.

\bibitem[FPS07]{FSP}
Ronald Fintushel, B.~Doug Park, and Ronald~J. Stern.
\newblock Reverse engineering small 4-manifolds.
\newblock {\em Algebr. Geom. Topol.}, 7:2103--2116, 2007.

\bibitem[Fre82]{Freedman}
Michael~Hartley Freedman.
\newblock The topology of four-dimensional manifolds.
\newblock {\em J. Differential Geometry}, 17(3):357--453, 1982.

\bibitem[FS84]{FS}
Ronald Fintushel and Ronald~J. Stern.
\newblock A {$\mu$}-invariant one homology {$3$}-sphere that bounds an orientable rational ball.
\newblock In {\em Four-manifold theory ({D}urham, {N}.{H}., 1982)}, volume~35 of {\em Contemp. Math.}, pages 265--268. Amer. Math. Soc., Providence, RI, 1984.

\bibitem[FS09]{SixLectures}
Ronald Fintushel and Ronald~J. Stern.
\newblock Six lectures on four 4-manifolds.
\newblock In {\em Low dimensional topology}, volume~15 of {\em IAS/Park City Math. Ser.}, pages 265--315. Amer. Math. Soc., Providence, RI, 2009.

\bibitem[HL12]{HoLi}
Chung-I Ho and Tian-Jun Li.
\newblock Luttinger surgery and {K}odaira dimension.
\newblock {\em Asian J. Math.}, 16(2):299--318, 2012.

\bibitem[Kaw09]{Kawauchi}
Akio Kawauchi.
\newblock Rational-slice knots via strongly negative-amphicheiral knots.
\newblock {\em Commun. Math. Res.}, 25(2):177--192, 2009.

\bibitem[Klu21]{Klug}
Michael~R. Klug.
\newblock A relative version of {R}ochlin's theorem, 2021.
\newblock arXiv:2011.12418.

\bibitem[KM07]{MonopoleBook}
Peter Kronheimer and Tomasz Mrowka.
\newblock {\em Monopoles and three-manifolds}, volume~10 of {\em New Mathematical Monographs}.
\newblock Cambridge University Press, Cambridge, 2007.

\bibitem[Lev23]{Levine}
Adam~Simon Levine.
\newblock A note on rationally slice knots.
\newblock {\em New York J. Math.}, 29:1363--1372, 2023.

\bibitem[Li06]{LiKodaira}
Tian-Jun Li.
\newblock Symplectic 4-manifolds with {K}odaira dimension zero.
\newblock {\em J. Differential Geom.}, 74(2):321--352, 2006.

\bibitem[LLP23]{LLP}
Adam~Simon Levine, Tye Lidman, and Lisa Piccirillo.
\newblock New constructions and invariants of closed exotic 4-manifolds, 2023.
\newblock arXiv:2307.08130.

\bibitem[Lut95]{Luttinger}
Karl~Murad Luttinger.
\newblock Lagrangian tori in {${\bf R}^4$}.
\newblock {\em J. Differential Geom.}, 42(2):220--228, 1995.

\bibitem[MMP24]{MMP}
Ciprian Manolescu, Marco Marengon, and Lisa Piccirillo.
\newblock Relative genus bounds in indefinite four-manifolds.
\newblock {\em Math. Ann.}, 390(1):1481--1506, 2024.

\bibitem[OS04]{OSSymplectic}
Peter Ozsv\'ath and Zolt\'an Szab\'o.
\newblock Holomorphic triangle invariants and the topology of symplectic four-manifolds.
\newblock {\em Duke Math. J.}, 121(1):1--34, 2004.

\bibitem[OS06]{OSTriangles}
Peter Ozsv\'ath and Zolt\'an Szab\'o.
\newblock Holomorphic triangles and invariants for smooth four-manifolds.
\newblock {\em Adv. Math.}, 202(2):326--400, 2006.

\bibitem[PSS05]{JSS}
Jongil Park, Andr\'as~I. Stipsicz, and Zolt\'an Szab\'o.
\newblock Exotic smooth structures on {$\Bbb{CP}^2\#5\overline{\Bbb{CP}^2}$}.
\newblock {\em Math. Res. Lett.}, 12(5-6):701--712, 2005.

\bibitem[Ras10]{Rasmussen}
Jacob Rasmussen.
\newblock Khovanov homology and the slice genus.
\newblock {\em Invent. Math.}, 182(2):419--447, 2010.

\bibitem[SS23]{StipsiczSzabo}
András~I. Stipsicz and Zoltán Szabó.
\newblock On the minimal genus problem in four-manifolds, 2023.
\newblock arXiv:2307.04202.

\bibitem[Thu76]{Thurston}
W.~P. Thurston.
\newblock Some simple examples of symplectic manifolds.
\newblock {\em Proc. Amer. Math. Soc.}, 55(2):467--468, 1976.

\bibitem[Wal64]{Wall}
C.~T.~C. Wall.
\newblock Diffeomorphisms of {$4$}-manifolds.
\newblock {\em J. London Math. Soc.}, 39:131--140, 1964.

\end{thebibliography}

\end{document}